\documentclass[british]{article}
\usepackage[T1]{fontenc}
\usepackage[latin9]{luainputenc}
\usepackage{amsmath}
\usepackage{amsthm}
\usepackage{amssymb}

\makeatletter
\theoremstyle{plain}
\newtheorem{thm}{\protect\theoremname}
\theoremstyle{remark}
\newtheorem{rem}[thm]{\protect\remarkname}
\theoremstyle{definition}
\newtheorem{defn}[thm]{\protect\definitionname}
\theoremstyle{plain}
\newtheorem{lem}[thm]{\protect\lemmaname}
\theoremstyle{plain}
\newtheorem{prop}[thm]{\protect\propositionname}
\theoremstyle{plain}
\newtheorem{cor}[thm]{\protect\corollaryname}

\@ifundefined{date}{}{\date{}}

\usepackage{aecompl}

\usepackage{float}

\usepackage{babel}
\providecommand{\corollaryname}{Corollary}
  \providecommand{\definitionname}{Definition}
  \providecommand{\lemmaname}{Lemma}
  \providecommand{\propositionname}{Proposition}
  \providecommand{\remarkname}{Remark}
\providecommand{\theoremname}{Theorem}

\newcommand{\IP}{\mathbb{P}}

\usepackage{babel}
\providecommand{\corollaryname}{Corollary}
  \providecommand{\lemmaname}{Lemma}
  \providecommand{\propositionname}{Proposition}
  \providecommand{\remarkname}{Remark}
\providecommand{\theoremname}{Theorem}

\usepackage{babel}
\providecommand{\corollaryname}{Corollary}
  \providecommand{\lemmaname}{Lemma}
  \providecommand{\propositionname}{Proposition}
  \providecommand{\remarkname}{Remark}
\providecommand{\theoremname}{Theorem}

\makeatother

\usepackage{babel}
\providecommand{\corollaryname}{Corollary}
\providecommand{\definitionname}{Definition}
\providecommand{\lemmaname}{Lemma}
\providecommand{\propositionname}{Proposition}
\providecommand{\remarkname}{Remark}
\providecommand{\theoremname}{Theorem}

\begin{document}
\title{Two Groups in a Curie-Weiss Model}
\author{Werner Kirsch and Gabor Toth\\
 Fakult\"at f\"ur Mathematik und Informatik \\
 FernUniversit\"at Hagen, Germany}
\maketitle

\section{Introduction}

The Curie-Weiss model is probably the easiest model of magnetism which
shows a phase transition between a diamagnetic and a ferromagnetic
phase. In this model the spins can take values in $\{-1,1\}$ (or
up/down), each spin interacts with all the others in the same way.
More precisely, for finitely many spins $(X_{1},X_{2},\ldots,X_{N})\in\{-1,1\}$
the energy of the spins is given by 
\begin{align}
H~=~H(X_{1},\ldots,X_{N})~:=~-\frac{J}{2N}\,\big(\sum_{j=1}^{N}\,X_{j}\big)^{2}\,,
\end{align}
where $J$ is a positive real number.

Consequently, in the `canonical ensemble' with inverse temperature
$\beta\geq0$ the probability of a spin configuration is given by
\begin{align}
\IP\big(X_{1}=x_{1},\ldots,X_{N}=x_{N}\big)~:=~Z^{-1}\;e^{-\beta H(x_{1},\ldots,x_{N})}\label{eq:Pbeta}
\end{align}
where $x_{i}\in\{-1,1\}$ and $Z$ is a normalization constant which
depends on $N$, $J$ and $\beta$. Since only the product of $\beta$
and $J$ occurs in \eqref{eq:Pbeta} we may set $J=1$ without loss
of generality.

The quantity 
\begin{align}
S_{N}~=~\sum_{j=1}^{N}X_{j}
\end{align}
is called the (total) magnetization. It is well known (see e.\,g.
Ellis \cite{Ellis} or \cite{key-1}) that the Curie-Weiss model has
a phase transition at $\beta=1$ in the following sense 
\begin{align}
\frac{1}{N}\,S_{N}~\Longrightarrow~\frac{1}{2}\,(\delta_{-m(\beta)}+\delta_{m(\beta)})\label{eq:lln}
\end{align}
where $\Rightarrow$ denotes convergence in distribution, $\delta_{x}$
the Dirac measure in $x$.

For $\beta\leq1$ we have $m(\beta)=0$ which is the unique solution
of 
\begin{align}
\tanh(\beta x)=x\label{eq:mbeta}
\end{align}
for this case.

If $\beta>1$ equation \eqref{eq:mbeta} has exactly three solutions
and $m(\beta)$ is the unique positive one.

Equation \eqref{eq:lln} is a substitute for the law of large numbers
for i.i.d. random variables.

Moreover, for $\beta<1$ there is a central limit theorem, i.\,e.
\begin{align}
\frac{1}{\sqrt{N}}\,S_{N}~\Longrightarrow~\mathcal{N}(0,\frac{1}{1-\beta})
\end{align}

For $\beta=1$ there is no such central limit theorem. In fact, the
random variables 
\begin{align}
\frac{1}{N^{3/4}}\,S_{N}
\end{align}
converge in distribution to a limit which is not a normal distribution.

In this paper we form out of $N$ Curie-Weiss spins two disjoint groups
$X_{1},\ldots,X_{N_{1}}$ and $Y_{1},\ldots,Y_{N_{2}}$ with $N_{1}+N_{2}\leq N$.
We let $N_{1}$ and $N_{2}$ depend on $N$ in such a way that both
$N_{1}$ and $N_{2}$ go to infinity as $N$ does. We consider the
asymptotic behaviour of the two-dimensional random variables 
\begin{equation}
\big(\sum_{i=1}^{N_{1}}\,X_{i}\,,\,\sum_{j=1}^{N_{2}}\,Y_{j}\,\big)
\end{equation}
as $N$ goes to infinity.

We prove 
\begin{thm}[Law of Large Numbers]
\label{LLN_hom} If $N_{1},N_{2}\to\infty$ as $N\to\infty$, then
we have for all $\beta$ 
\begin{equation}
\big(\frac{1}{N_{1}}\sum_{i=1}^{N_{1}}X_{i},\frac{1}{N_{2}}\sum_{j=1}^{N_{2}}Y_{j}\big)~\underset{N\to\infty}{\Longrightarrow}~\frac{1}{2}\big(\delta_{(-m(\beta),-m(\beta))}+\delta_{(m(\beta),m(\beta))}\big).\label{eq:LLNhom}
\end{equation}
\end{thm}

Above `$\Longrightarrow$' denotes convergence in distribution of
the $2$-dimensional random variable on the left hand side. 
\begin{rem}
If we consider a model without interaction between the groups $X_{i}$
and $Y_{j}$ then the limit in \eqref{eq:LLNhom} is 
\[
\frac{1}{4}\big(\delta_{(-m(\beta),-m(\beta))}+\delta_{(-m(\beta),m(\beta))}+\delta_{(m(\beta),-m(\beta))}+\delta_{(m(\beta),m(\beta))}\big)
\]
\end{rem}

For $\beta<1$ we also have a central limit theorem. The covariance
of the limiting normal distribution depends on the growth rate of
$N_{1}$ and $N_{2}$. We set 
\begin{equation}
\alpha_{1}=\lim\frac{N_{1}}{N}\qquad\qquad\alpha_{2}=\lim_{N\to\infty}\frac{N_{2}}{N}
\end{equation}
and assume that these limits exist. 
\begin{thm}[Central Limit Theorem]
\label{CLT_hom} If $\beta<1$, then 
\begin{equation}
(\frac{1}{\sqrt{N_{1}}}\sum_{i=1}^{N_{1}}X_{i},\frac{1}{\sqrt{N_{2}}}\sum_{j=1}^{N_{2}}Y_{j})~\underset{N\to\infty}{\Longrightarrow}~\mathcal{N}\big((0,0),C),\label{eq:CLT}
\end{equation}
where the covariance matrix $C$ is given by 
\begin{align}
C & =\left[\begin{array}{cc}
1+\alpha_{1}\frac{\beta}{1-\beta} & \sqrt{\alpha_{1}\alpha_{2}}\frac{\beta}{1-\beta}\\
\sqrt{\alpha_{1}\alpha_{2}}\frac{\beta}{1-\beta} & 1+\alpha_{2}\frac{\beta}{1-\beta}
\end{array}\right]
\end{align}
\end{thm}

In particular, for sublinear growth of either $N_{1}$ or $N_{2}$,
i.\,e. if $\alpha_{1}=0$ or $\alpha_{2}=0$, the standardized sums
in \eqref{eq:CLT} are asymptotically independent.

We mention that the Curie-Weiss model is also used to model the behaviour
of voters who have the choice to vote `Yea' (spin=1, say) or `Nay'
(spin=-1) (see \cite{key-2}).

In the proof of both results we employ the moment method (see e.\,g.
\cite{Breiman} or \cite{key-1}). Thus, to show the convergence in
distribution of a sequence $(X_{n},Y_{n})$ of two-dimensional random
variables to a measure $\mu$ on $\mathbb{R}^{2}$ we prove that 
\begin{align}
\mathbb{E}\Big(X_{n}^{K}\cdot Y_{n}^{L}\Big)~\longrightarrow~\int x^{K}y^{L}\,\mu(dx,dy)\label{eq:momcon}
\end{align}
for all $K,L\in\mathbb{N}$.

Equation \eqref{eq:momcon} implies convergence in distribution if
the moments of $\mu$ grow only moderately, namely if for some constant
$A$ and $C$ and all $K,L$ 
\begin{equation}
\int|x|^{K}|y|^{L}\,\mu(dx,dy)~\leq~A\,C^{K+L}\,(K+L)!
\end{equation}
holds.

Some time after publishing the first version of this paper on arXiv,
we became aware of the articles \cite{FM} and \cite{FC} which contain
the above results as special cases. The methods used by those authors
is very different from ours. We are grateful to Francesca Collet for
drawing our attention to the papers \cite{FM} and \cite{FC}.

\section{Preparation}

To use the method of moments we have to evaluate sums of the form
\begin{align}
 & \mathbb{E}\left[\left(\sum_{i=1}^{N_{1}}X_{i}\right)^{K}\left(\sum_{j=1}^{N_{2}}Y_{j}\right)^{L}\right]\nonumber \\
=~ & \sum_{i_{1},\ldots,i_{K}}\sum_{j_{1},\ldots,j_{L}}\,\mathbb{E}\Big(X_{i_{1}}\cdot X_{i_{2}}\cdot\ldots\cdot X_{i_{K}}\cdot Y_{j_{1}}\cdot Y_{j_{2}}\cdot\ldots\cdot Y_{j_{L}}\Big)\,.
\end{align}
To do the book-keeping for these huge sums we introduce a few combinatorial
concepts taken from \cite{key-1}.
\begin{defn}
We define a multiindex $\underline{i}=(i_{1},i_{2},\ldots,i_{L})\in\{1,2,\ldots,N\}^{L}$. 
\begin{enumerate}
\item For $j\in\{1,2,\ldots,N\}$ we set\\
 
\[
\nu_{j}(\underline{i})=|\{k\in\{1,2,\ldots,L\}|i_{k}=j\}|,
\]
where $|M|$ denotes the number of elements in the set $M$. 
\item For $l=0,1,\ldots,L$ we define\\
 
\[
\rho_{l}(\underline{i})=|\{j|\nu_{j}(\underline{i})=l\}|
\]
and\\
 
\[
\underline{\rho}(\underline{i})=(\rho_{1}(\underline{i}),\ldots,\rho_{L}(\underline{i})).
\]
\end{enumerate}
\end{defn}

The numbers $\nu_{j}(\underline{i})$ represent the multiplicity of
each index $j\in\{1,2,\ldots,N\}$ in the multiindex $\underline{i}$,
and $\rho_{l}(\underline{i})$ represents the number of indices in
$\underline{i}$ that occur exactly $l$ times. We shall call such
$\underline{\rho}(\underline{i})$ profile vectors. 
\begin{lem}[Lemma 3.8]
\label{lem:sum-l-rho}For all $\underline{i}=(i_{1},i_{2},\ldots,i_{L})\in\{1,2,\ldots,N\}^{L}$
we have\\
 $\sum_{l=1}^{L}l\rho_{l}(\underline{i})=L$. 
\end{lem}

\begin{defn}
Let $\underline{r}=(r_{1},\ldots,r_{L})$, $\sum_{l=1}^{L}lr_{l}=L$,
be a profile vector. We define

\[
w_{L}(\underline{r})=|\{\underline{i}\in\{1\ldots,N\}^{L}|\underline{\rho}(\underline{i})=\underline{r}\}|
\]
to represent the number of multiindices $\underline{i}$ that have
a given profile vector $\underline{r}$. 
\end{defn}

We now define the set of all profile vectors for a given $L\in\mathbb{N}$. 
\begin{defn}
Let $\Pi^{(L)}=\{\underline{r}\in\{0,1,\ldots,L\}^{L}|\sum_{l=1}^{L}lr_{l}=L\}$.
Some important subsets of $\Pi^{(L)}$ are $\Pi_{k}^{(L)}=\{\underline{r}\in\Pi^{(L)}|r_{1}=k\}$,
$\Pi^{0(L)}=\{\underline{r}\in\Pi^{(L)}|r_{l}=0\text{ for all }l\geq3\}$
and $\Pi^{+(L)}=\{\underline{r}\in\Pi^{(L)}|r_{l}>0\text{ for some }l\geq3\}$. 
\end{defn}

\begin{prop}
\label{thm:comb-coeff-multiindex}For $\underline{r}\in\Pi^{(L)}$
set $r_{0}=N-\sum_{l=1}^{L}r_{l}$, then\\
 
\[
w_{L}(\underline{r})=\frac{N!}{r_{1}!r_{2}!\ldots r_{L}!r_{0}!}\frac{L!}{1!^{r_{1}}2!^{r_{2}}\cdots L!^{r_{L}}}.
\]
\end{prop}

\section{Proofs \label{sec:Hom}}

\subsection{Law of Large Numbers}

We are interested in the behaviour of the partial sums

\[
\frac{1}{N_{1}}\sum_{i=1}^{N_{1}}X_{i},\frac{1}{N_{2}}\sum_{j=1}^{N_{2}}Y_{j}.
\]

Suppose $K,L\in\mathbb{N}$. We want to calculate the moment

\begin{equation}
\mathbb{E}\left[\left(\frac{1}{N_{1}}\sum_{i=1}^{N_{1}}X_{i}\right)^{K}\left(\frac{1}{N_{2}}\sum_{j=1}^{N_{2}}Y_{j}\right)^{L}\right].\label{eq:LLN_MKL}
\end{equation}

We distinguish between multiindices that have a repeated index and
those that do not. The following proposition shows that only the multiindices
$\underline{i}\in\prod^{(K)},\underline{j}\in\prod^{(L)}$ in which
each index occurs exactly once contribute asymptotically to the moments

\begin{equation}
\frac{1}{N_{1}^{K}N_{2}^{L}}\sum_{\underline{i}\in\Pi^{(K)}}\sum_{\underline{j}\in\Pi^{(L)}}w_{K}(\underline{i})w_{L}(\underline{j})\mathbb{E}(X_{\underline{i}}Y_{\underline{j}}).\label{eq:moments_beta_greater_1}
\end{equation}

\begin{prop}
If $\underline{i}\in\prod^{(K)}$ or $\underline{j}\in\prod^{(L)}$
has an index that occurs more than once, then $\frac{1}{N_{1}^{K}N_{2}^{L}}w_{K}(\underline{i})w_{L}(\underline{j})\mathbb{E}(X{}_{\underline{i}}Y{}_{\underline{j}})$
converges to $0$ as $N\rightarrow\infty$. 
\end{prop}

\begin{proof}
We shall assume without loss of generality that $\underline{i}$ contains
a repeated index. This implies that the profile vector $\underline{r}=(r_{1},\ldots,r_{K})=\underline{\rho}(\underline{i})$
has some $r_{j}>0$ where $j>1$. For multiindex $\underline{j}$,
define the profile vector $\underline{s}=(s_{1},\ldots,s_{L})$. There
are

\begin{align*}
w_{K}(\underline{r}) & =\frac{1}{r_{1}!\cdots r_{K}!}\frac{K!}{1!^{r_{1}}\cdots K!^{r_{K}}}N_{1}^{\sum_{k=1}^{K}r_{k}}\\
 & =C_{K}N_{1}^{\sum_{k=1}^{K}r_{k}}
\end{align*}
and

\begin{align*}
w_{L}(\underline{s}) & =\frac{1}{s_{1}!\cdots s_{K}!}\frac{L!}{1!^{s_{1}}\cdots K!^{s_{L}}}N_{1}^{\sum_{l=1}^{L}s_{l}}\\
 & =C_{L}N_{1}^{\sum_{l=1}^{L}s_{l}}
\end{align*}

such multiindices. Note that $\sum_{k=1}^{K}r_{k}<K$ due to the existence
of $r_{j}>0$ for some $j>1$, as well as $\sum_{k=1}^{K}kr_{k}=K$,
and $\sum_{l=1}^{L}s_{l}\leq L$. Hence

\begin{align*}
\frac{1}{N_{1}^{K}N_{2}^{L}}w_{K}(\underline{i})w_{L}(\underline{j}) & =\frac{1}{N_{1}^{K}N_{2}^{L}}C_{K}N_{1}^{\sum_{k=1}^{K}r_{k}}C_{L}N_{1}^{\sum_{l=1}^{L}s_{l}}\\
 & \leq\frac{1}{N_{1}^{K}N_{2}^{L}}C_{K}C_{L}N_{1}^{K-1}N_{2}^{L}\\
 & =C_{K}C_{L}N_{1}^{-1}\\
 & \rightarrow0
\end{align*}
and the assertion holds. 
\end{proof}
Since multiindices with repeated indices do not contribute asymptotically
to the moment in \eqref{eq:LLN_MKL}, that leaves us with all those
multiindices that do not contain the same index more than once. Of
these there are asymptotically $N_{1}^{K}$ in $\Pi^{(K)}$ and $N_{2}^{L}$
in $\Pi^{(L)}$. Hence the moment is asymptotically given by

\[
\frac{1}{N_{1}^{K}N_{2}^{L}}N_{1}^{K}N_{2}^{L}\mathbb{E}(X_{\underline{i}}Y_{\underline{j}})=\mathbb{E}(X_{\underline{i}}Y_{\underline{j}}).
\]
So the moment is asymptotically equal to the correlations $\mathbb{E}(X_{\underline{i}}Y_{\underline{j}})$.
As shown in \cite{key-1}, this expression depends on the value of
$\beta$. We have for $\beta\leq1$

\[
\mathbb{E}(X_{\underline{i}}Y_{\underline{j}})\rightarrow0,
\]
as $N\rightarrow\infty$. For $\beta>1$, however,

\[
\mathbb{E}(X_{\underline{i}}Y_{\underline{j}})\approx m(\beta)^{K+L}.
\]

We can conclude, that as in the case of a single group in the Curie-Weiss
model, the law of large numbers (theorem \ref{LLN_hom}) holds and
the random vectors $(\frac{1}{N_{1}^{K}}\sum_{i=1}^{N_{1}}X_{i},\frac{1}{N_{2}^{L}}\sum_{j=1}^{N_{2}}Y_{j})_{N}$
converge to the random vector $\frac{1}{2}(\delta_{-(m(\beta),m(\beta))}+\delta_{(m(\beta),m(\beta))})$.

\subsection{Central Limit Theorem}

Now, we are interested in the behaviour of the partial sums

\[
\frac{1}{\sqrt{N_{1}}}\sum_{i=1}^{N_{1}}X_{i},\frac{1}{\sqrt{N_{2}}}\sum_{j=1}^{N_{2}}Y_{j}.
\]
Suppose $K,L\in\mathbb{N}$. We want to calculate the moment

\begin{equation}
\mathbb{E}\left[\left(\frac{1}{\sqrt{N_{1}}}\sum_{i=1}^{N_{1}}X_{i}\right)^{K}\left(\frac{1}{\sqrt{N_{2}}}\sum_{j=1}^{N_{2}}Y_{j}\right)^{L}\right].\label{eq:CLT_MKL}
\end{equation}

Since within each group the $X_{i},Y_{j}$ are exchangeable, only
the number of indices that occur $0,1,...$ times matters. Hence we
write for all $(i_{1},\ldots,i_{K})$ and all $(j_{1},\ldots,j_{L})$
with $X(\underline{i})=X_{i_{1}}X_{i_{2}}\cdots X_{i_{K}}$ and $Y(\underline{j})=Y_{j_{1}}Y_{j_{2}}\cdots Y_{j_{L}}$,
provided that the multiindices $(i_{1},\ldots,i_{K})$ and $(j_{1},\ldots,j_{L})$
have the profile vectors $\underline{r}$ and $\underline{s}$, respectively.

The moments in (\ref{eq:CLT_MKL}) are thus given by

\begin{equation}
\frac{1}{N_{1}^{K/2}N_{2}^{L/2}}\sum_{k=0}^{K}\sum_{\underline{i}\in\Pi_{k}^{(K)}}\sum_{l=0}^{L}\sum_{\underline{j}\in\Pi_{l}^{(L)}}w_{K}(\underline{i})w_{L}(\underline{j})\mathbb{E}(X_{\underline{i}}Y_{\underline{j}}).\label{eq:sum}
\end{equation}

The value of the above depends on the value of the inverse temperature
parameter $\beta$.

We separate the above sum into four summands:

\begin{align*}
A_{1} & =\frac{1}{N_{1}^{K/2}N_{2}^{L/2}}\sum_{k=0}^{K}\sum_{\underline{i}\in\Pi_{k}^{0(K)}}\sum_{l=0}^{L}\sum_{\underline{j}\in\Pi_{l}^{0(L)}}w_{K}(\underline{i})w_{L}(\underline{j})\mathbb{E}(X_{\underline{i}}Y_{\underline{j}}),\\
A_{2} & =\frac{1}{N_{1}^{K/2}N_{2}^{L/2}}\sum_{k=0}^{K}\sum_{\underline{i}\in\Pi_{k}^{0(K)}}\sum_{l=0}^{L}\sum_{\underline{j}\in\Pi_{l}^{+(L)}}w_{K}(\underline{i})w_{L}(\underline{j})\mathbb{E}(X_{\underline{i}}Y_{\underline{j}}),\\
A_{3} & =\frac{1}{N_{1}^{K/2}N_{2}^{L/2}}\sum_{k=0}^{K}\sum_{\underline{i}\in\Pi_{k}^{+(K)}}\sum_{l=0}^{L}\sum_{\underline{j}\in\Pi_{l}^{0(L)}}w_{K}(\underline{i})w_{L}(\underline{j})\mathbb{E}(X_{\underline{i}}Y_{\underline{j}}),\\
A_{4} & =\frac{1}{N_{1}^{K/2}N_{2}^{L/2}}\sum_{k=0}^{K}\sum_{\underline{i}\in\Pi_{k}^{+(K)}}\sum_{l=0}^{L}\sum_{\underline{j}\in\Pi_{l}^{+(L)}}w_{K}(\underline{i})w_{L}(\underline{j})\mathbb{E}(X_{\underline{i}}Y_{\underline{j}}).
\end{align*}

We will show that, asymptotically, only $A_{1}$ contributes to the
sum (\ref{eq:sum}). 
\begin{prop}
The limit of $A_{2}$ as $N_{2}$ goes to infinity is 0. 
\end{prop}

\begin{proof}
For fixed $k,l$ let $\underline{i}\in\Pi_{k}^{0(K)}$ and $\underline{j}\in\Pi_{l}^{+(L)}$.
Then

\begin{align*}
\frac{1}{N_{1}^{K/2}N_{2}^{L/2}}w_{K}(\underline{i})w_{L}(\underline{j})\mathbb{E}(X_{\underline{i}}Y_{\underline{j}}) & \approx\\
\frac{1}{N_{1}^{K/2}N_{2}^{L/2}}\frac{N_{1}!}{r_{1}!r_{2}!\cdots r_{K}!r_{0}!}\frac{K!}{1!^{r_{1}}2!^{r_{2}}\cdots K!^{r_{K}}}\frac{N_{2}!}{s_{1}!s_{2}!\cdots s_{L}!s_{0}!} & \cdot\\
\cdot\frac{L!}{1!^{s_{1}}2!^{s_{2}}\cdots L!^{s_{L}}}(k+l-1)!!\left(\frac{\beta}{1-\beta}\right)^{\frac{k+l}{2}}N^{-\frac{k+l}{2}} & \approx\\
\frac{c}{N_{2}^{L/2}}N_{2}^{\sum_{j=1}^{L}s_{j}}N^{-\frac{k+l}{2}} & \leq\\
\frac{c}{N_{2}^{L/2}}N_{2}^{\frac{L+l-1}{2}}N^{-\frac{k+l}{2}} & \approx\\
c\alpha_{1}^{k/2}\alpha_{2}^{l/2}N_{2}^{-\frac{k+1}{2}} & \rightarrow0.
\end{align*}
The constant $c$ in the fourth line above represents the product
of all those factors that do not depend on $N_{2}$ and $N$. In the
inequality above we used 
\begin{itemize}
\item $\sum_{j=1}^{L}s_{j}\leq l+\frac{1}{2}\sum_{j=2}^{L}js_{j}-\frac{1}{2}=\frac{l}{2}+\frac{1}{2}\sum_{j=1}^{L}js_{j}-\frac{1}{2}=\frac{L+l-1}{2}$
and 
\item $N_{2}\approx N$ for large $N_{2}$ and fixed $N_{1}$. 
\end{itemize}
Since each summand goes to zero, and there are only finitely many
summands, $A_{2}$ goes to zero as $N_{2}$ goes to infinity. 
\end{proof}
Note that although we assumed a fixed $N_{1}$ in both the statement
and the proof, the assertion would also hold if we assumed $N_{\nu}=\alpha_{\nu}N$
for both groups and let $N$ go to infinity. 
\begin{prop}
\label{prop:A24}The following limits hold: 
\begin{enumerate}
\item The limit of $A_{3}$ as $N_{1}$ goes to infinity is 0. 
\item The limit of $A_{4}$ as $N_{1}$ goes to infinity is 0. 
\item The limit of $A_{4}$ as $N_{2}$ goes to infinity is 0. 
\end{enumerate}
\end{prop}

The proof of this proposition is very similar to the previous one,
so we shall omit it.

Using these propositions, we obtain 
\begin{cor}
Asymptotically, i.e. if both $N_{1}$ and $N_{2}$ go to infinity,
we have $\frac{1}{N_{1}^{K/2}N_{2}^{L/2}}\sum_{r=0}^{K}\sum_{\underline{i}\in\Pi_{r}^{(K)}}\sum_{s=0}^{L}\sum_{\boldsymbol{\underline{j}}\in\Pi_{s}^{(L)}}w_{K}(\underline{i})w_{L}(\underline{j})\mathbb{E}(X_{\underline{i}}Y_{\underline{j}})\approx\lim_{N_{1}\rightarrow\infty,N_{2}\rightarrow\infty}A_{1}$. 
\end{cor}

Our task is, therefore, to calculate this limit.

Proceeding along the same lines as in the proof of Theorem 5.23, we
note that $A_{1}$ is asymptotically 0 if $r+s$ is odd. The reason
for that is Theorem 5.17, which states that for $r+s$ odd $\mathbb{E}(X_{\underline{i}}Y_{\underline{j}})=\mathbb{E}(X_{11}X_{12}\cdots X_{1,r+s})=0$.
(We used the fact, that in $A_{1}$ both $\underline{i}$ and $\underline{j}$
are such multiindices that each index occurs once or twice. Hence
we can ignore all indices that occur more than once.)\\
 Now note that $k+2r_{2}=K$ and $l+2s_{2}=L$. This implies that
$K$ must have the same parity as $k$ and $L$ the same as $l$.
Hence $K+L$ must be even as well. We have to distinguish two cases
here: 
\begin{enumerate}
\item $K,L$ even, 
\item $K,L$ odd. 
\end{enumerate}
If both $K$ and $L$ are even, we can express $A_{1}$ as

\begin{align*}
A_{1} & =\frac{1}{N_{1}^{K/2}N_{2}^{L/2}}\sum_{k=0}^{K/2}\sum_{l=0}^{L/2}w_{K}(2k)w_{L}(2l)\mathbb{E}(X(\underline{\boldsymbol{r}})Y(\underline{\boldsymbol{s}}))\\
 & \approx\frac{1}{N_{1}^{K/2}N_{2}^{L/2}}\sum_{k=0}^{K/2}\sum_{l=0}^{L/2}N_{1}^{K/2+k}\frac{K!}{(2k)!(K/2-k)!2^{K/2-k}}N_{2}^{L/2+l}\cdot\\
 & \cdot\frac{L!}{(2l)!(L/2-l)!2^{L/2-l}}(2(k+l)-1)!!\left(\frac{\beta}{1-\beta}\right)^{k+l}N^{-(k+l)}\\
 & \approx\frac{K!L!}{(\frac{K}{2})!(\frac{L}{2})!2^{\frac{K+L}{2}}}\sum_{k=0}^{K/2}\sum_{l=0}^{L/2}\alpha_{1}^{r}\alpha_{2}^{s}\frac{(\frac{K}{2})!}{(\frac{K}{2}-k)!(2k)!2^{-k}}\cdot\\
 & \cdot\frac{(\frac{L}{2})!}{(\frac{L}{2}-l)!(2l)!2^{-l}}(2(k+l)-1)!!\left(\frac{\beta}{1-\beta}\right)^{k+l}\\
 & =(K-1)!!(L-1)!!\sum_{k=0}^{K/2}\sum_{l=0}^{L/2}\alpha_{1}^{k}\alpha_{2}^{l}\frac{(\frac{K}{2})!}{(\frac{K}{2}-k)!k!}\frac{(\frac{L}{2})!}{(\frac{L}{2}-l)!l!}\cdot\\
 & \cdot\frac{(2(k+l)-1)!!}{(2k-1)!!(2l-1)!!}\left(\frac{\beta}{1-\beta}\right)^{k+l}\\
 & =(K-1)!!(L-1)!!\sum_{k=0}^{K/2}\sum_{l=0}^{L/2}\alpha_{1}^{k}\alpha_{2}^{l}\left(\begin{array}{c}
K/2\\
k
\end{array}\right)\left(\begin{array}{c}
L/2\\
l
\end{array}\right)\cdot\\
 & \cdot\frac{\frac{(2(k+l))!}{(k+l)!2^{k+l}}}{\frac{(2k)!}{k!2^{k}}\frac{(2l)!}{l!2^{l}}}\left(\frac{\beta}{1-\beta}\right)^{k+l}\\
 & =(K-1)!!(L-1)!!\sum_{k=0}^{K/2}\sum_{l=0}^{L/2}\alpha_{1}^{k}\alpha_{2}^{l}\left(\begin{array}{c}
K/2\\
k
\end{array}\right)\left(\begin{array}{c}
L/2\\
l
\end{array}\right)\cdot\\
 & \cdot\frac{\left(\begin{array}{c}
2(k+l)\\
2k
\end{array}\right)}{\left(\begin{array}{c}
k+l\\
k
\end{array}\right)}\left(\frac{\beta}{1-\beta}\right)^{k+l}.
\end{align*}

The case where $K,L$ are odd is similar to the above.

\subsubsection{Linear Population Growth}

We now show the central limit theorem for two groups in a Curie-Weiss
model (theorem \ref{CLT_hom}). Let $\beta<1$ be the inverse temperature
parameter, and define for convenience $\bar{\beta}=\frac{\beta}{1-\beta}$.
Let $N$ be the overall size of the population and assume that for
$\alpha_{1},\alpha_{2}>0,\alpha_{1}+\alpha_{2}\leq1,$ $N_{1}\approx\alpha_{1}N$
and $N_{2}\approx\alpha_{2}N$ are two groups within this population.
We shall use the symbols $X_{i}$ and $Y_{j}$ to refer to individual
votes within groups 1 and 2, respectively. Define the normed sums
$S_{1}=\frac{1}{\sqrt{N_{1}}}\sum_{i=1}^{N_{1}}X_{i}$ and $S_{2}=\frac{1}{\sqrt{N_{2}}}\sum_{j=1}^{N_{2}}Y_{j}$.
Using Isserlis's Theorem (see \cite{key-3}) and the recursive structure
it implies for the moments of a bivariate normal distribution, we
show that asymptotically the moments of the infinite sequence the
random vectors $(S_{1},S_{2})_{N}$ converge to those of a bivariate
normal distribution with zero mean and covariance matrix

\[
\left[\begin{array}{cc}
\mathbb{E}(S_{1}^{2}) & \mathbb{E}(S_{1}S_{2})\\
\mathbb{E}(S_{1}S_{2}) & \mathbb{E}(S_{1}^{2})
\end{array}\right]=\left[\begin{array}{cc}
1+\alpha_{1}\bar{\beta} & \sqrt{\alpha_{1}\alpha_{2}}\bar{\beta}\\
\sqrt{\alpha_{1}\alpha_{2}}\bar{\beta} & 1+\alpha_{2}\bar{\beta}
\end{array}\right].
\]

Let $(Z_{1},Z_{2})$ be such a bivariate normal random vector.

Isserlis's Theorem states that higher moments of the bivariate normal
distribution can be calculated

\[
E(Z_{1}^{K}Z_{2}^{L})=\sum_{\pi\in\mathcal{P}}\prod_{i=1}^{\frac{K+L}{2}}E(\xi_{i1}\xi_{i2}),
\]

where $\mathcal{P}$ is the set of all pair partitions on the set
$\{1,2,\ldots,K+L\}$ and $\xi_{i1}\xi_{i2}$ are two of the $K+L$
variables grouped together by a particular pair partition $\pi$.

We start by using Isserlis's Theorem to express higher moments of
the bivariate normal distribution recursively. 
\begin{lem}
For all $K,L\in\mathbb{N}_{0}$, the moments $m_{K,L}=E(Z_{1}^{K}Z_{2}^{L})$
satisfy the following equalities: 
\begin{enumerate}
\item $m_{K,L+2}=Km_{1,1}m_{K-1,L+1}+(L+1)m_{0,2}m_{K,L},$ 
\item $m_{K+2,L}=(K+1)m_{2,0}m_{K,L}+Lm_{1,1}m_{K+1,L-1}.$ 
\end{enumerate}
\end{lem}

If $K=0$ or $L=0$, then the formulas still hold, setting any moments
with negative indices equal to 0. Note that these two formulas suffice
to calculate any higher moment as a function of $K,L$ and $m_{2,0},m_{1,1},m_{0,2}$. 
\begin{proof}
We shall prove the first equation. By Isserlis's Theorem, we can calculate
$m_{K,L+2}$ as the sum over all permutations of product of expectations.
If we take away the two additional $Z_{2}$ variables, what remains
is a set of $K$ $Z_{1}$s and $L$ $Z_{2}$s. We have three possibilities
here: 
\begin{enumerate}
\item join one of the two additional $Z_{2}$s with one of the $K$ $Z_{1}$s. 
\item join one of the two additional $Z_{2}$s with one of the $L$ $Z_{2}$s. 
\item join the two additional $Z_{2}$s. 
\end{enumerate}
In the first case, we have $K$ times

\[
E(Z_{1}Z_{2})\sum_{\pi\in\mathcal{P}_{K-1,L+1}}\prod_{i=1}^{\frac{K+L}{2}}E(\xi_{i1}\xi_{i2})=m_{1,1}m_{K-1,L+1}.
\]
In the second case, we have $L$ times

\[
E(Z_{2}^{2})\sum_{\pi\in\mathcal{P}_{K,L}}\prod_{i=1}^{\frac{K+L}{2}}E(\xi_{i1}\xi_{i2})=m_{0,2}m_{K,L}.
\]
In the third case, we have $E(Z_{2}^{2})\sum_{\pi\in\mathcal{P}_{K,L}}\prod_{i=1}^{\frac{K+L}{2}}E(\xi_{i1}\xi_{i2})=m_{0,2}m_{K,L}.$
The result follows by summing over the three possible cases. 
\end{proof}
Let $M_{K,L}$ stand for the limit of the moments of the random vector
$(S_{1},S_{2})_{N}$. Our goal is to show that $M_{K,L}=m_{K,L}$
for all $K$ and $L$ and therefore the central limit theorem holds.
We accomplish this by showing that the two formulas in lemma 1 hold
for $M_{K,L}$ instead of $m_{K,L}$. Then, since by definition $M_{2,0}=m_{2,0},M_{1,1}=m_{1,1}$
and $M_{0,2}=m_{0,2}$, all higher moments must be equal, too.

In the two dimensional CW model, only $K+L$ has to be even for the
moment $M_{K,L}$ to be positive, allowing for two different cases
we need to consider: $K$ and $L$ being both even and $K$ and $L$
being both odd. Asymptotically, the moments are given by the formulas
$M_{K,L}=$

\[
\sum_{k=0}^{K/2}\sum_{l=0}^{L/2}\frac{K!}{(2k)!(K/2-k)!2^{K/2-k}}\frac{L!}{(2l)!(L/2-l)!2^{L/2-l}}(2(k+l)-1)!!\bar{\beta}^{k+l}\alpha_{1}^{k}\alpha_{2}^{l},
\]
when $K,L$ are even.

We shall now show that the first formula in lemma 1 holds for $M_{K,L}$,
assuming that $K,L$ are even. Then the left hand side of the formula
reads $M_{K,L+2}=$

\begin{align}
\sum_{k=0}^{K/2}\sum_{l=0}^{L/2+1}\frac{K!}{(2k)!(K/2-k)!2^{K/2-k}}\frac{(L+2)!}{(2l)!(L/2+1-l)!2^{L/2+1-l}}\cdot\nonumber \\
\cdot(2(k+l)-1)!!\bar{\beta}^{k+l}\alpha_{1}^{k}\alpha_{2}^{l} & .\label{eq:LHS}
\end{align}

On the right hand side, we have $KM_{1,1}M_{K-1,L+1}=$

\begin{align*}
K\sqrt{\alpha_{1}\alpha_{2}}\bar{\beta}\sum_{k=0}^{\frac{K}{2}-1}\sum_{l=0}^{\frac{L}{2}}\frac{(K-1)!}{(2k+1)!(\frac{K}{2}-1-k)!2^{\frac{K}{2}-1-k}}\frac{(L+1)!}{(2l+1)!(\frac{L}{2}-l)!2^{\frac{L}{2}-l}}\cdot\\
\cdot(2(k+l)+1)!!\bar{\beta}^{k+l+1}\alpha_{1}^{k+1/2}\alpha_{2}^{l+1/2},
\end{align*}
and $(L+1)M_{0,2}M_{K,L}=$

\begin{align*}
(L+1)(1+\alpha_{2}\bar{\beta})\sum_{k=0}^{K/2}\sum_{l=0}^{L/2}\frac{K!}{(2k)!(K/2-k)!2^{K/2-k}}\frac{L!}{(2l)!(L/2-l)!2^{L/2-l}}\cdot\\
\cdot(2(k+l)-1)!!\bar{\beta}^{k+l}\alpha_{1}^{k}\alpha_{2}^{l}.
\end{align*}

Multiplying by the factors $K\sqrt{\alpha_{1}\alpha_{2}}\bar{\beta}$
and $(L+1)(1+\alpha_{2}\bar{\beta})$ and separating the second term
into two sums, we obtain the following three summands on the right
hand side:

\begin{align}
\sum_{k=0}^{\frac{K}{2}-1}\sum_{l=0}^{\frac{L}{2}}\frac{K!}{(2k+1)!(\frac{K}{2}-1-k)!2^{\frac{K}{2}-1-k}}\frac{(L+1)!}{(2l+1)!(\frac{L}{2}-l)!2^{\frac{L}{2}-l}}\cdot\nonumber \\
\cdot(2(k+l)+1)!!\bar{\beta}^{k+l+2}\alpha_{1}^{k+1}\alpha_{2}^{l+1},\label{eq:RHS1}
\end{align}

\begin{align}
\sum_{k=0}^{K/2}\sum_{l=0}^{L/2}\frac{K!}{(2k)!(K/2-k)!2^{K/2-k}}\frac{(L+1)!}{(2l)!(L/2-l)!2^{L/2-l}}\cdot\nonumber \\
\cdot(2(k+l)-1)!!\bar{\beta}^{k+l+1}\alpha_{1}^{k}\alpha_{2}^{l+1},\label{eq:RHS2}
\end{align}

\begin{equation}
\sum_{k=0}^{K/2}\sum_{l=0}^{L/2}\frac{K!}{(2k)!(K/2-k)!2^{K/2-k}}\frac{(L+1)!}{(2l)!(L/2-l)!2^{L/2-l}}(2(k+l)-1)!!\bar{\beta}^{k+l}\alpha_{1}^{k}\alpha_{2}^{l}.\label{eq:RHS3}
\end{equation}

Note that the powers of $\alpha_{1}$ and $\alpha_{2}$ in (\ref{eq:LHS})
run through the sets $\{0,1,\ldots,K/2\}$ and $\{0,1,\ldots,L/2+1\}$,
respectively. Once both powers are chosen, the power of $\bar{\beta}$
is given by their sum.

We prove the theorem by showing that for each possible value of said
powers $k_{1}\in\{0,1,\ldots,K/2\}$ and $k_{2}\in\{0,1,\ldots,L/2+1\}$
the coefficient multiplying the term $\bar{\beta}^{k_{1}+k_{2}}\alpha_{1}^{k_{1}}\alpha_{2}^{k_{2}}$
is equal to the coefficient of the corresponding term on the right
hand side, given by the sum of (\ref{eq:RHS1}), (\ref{eq:RHS2})
and (\ref{eq:RHS3}).

Depending on the values of $k_{1}$ and $k_{2}$, not all of the three
sums on the right hand side contribute to the coefficient of $\bar{\beta}^{k_{1}+k_{2}}\alpha_{1}^{k_{1}}\alpha_{2}^{k_{2}}$
. It is only if $k_{1}\in\{1,\ldots,K/2\}$ and $k_{2}\in\{1,\ldots,L/2\}$
that all three sums contribute. That is the first case we want to
inspect.

The coefficient on the left hand side given by (\ref{eq:LHS}) is

\[
\frac{K!}{(2k_{1})!(K/2-k_{1})!2^{K/2-k_{1}}}\frac{(L+2)!}{(2k_{2})!(L/2+1-k_{2})!2^{L/2+1-k_{2}}}(2(k_{1}+k_{2})-1)!!.
\]

On the right hand side, we have three summands. The sum in (\ref{eq:RHS1})
contributes when $k=k_{1}-1$ and $l=k_{2}-1$:

\[
\frac{K!}{(2k_{1}-1)!(\frac{K}{2}-k_{1})!2^{\frac{K}{2}-k_{1}}}\frac{(L+1)!}{(2k_{2}-1)!(\frac{L}{2}+1-k_{2})!2^{\frac{L}{2}+1-k_{2}}}(2(k_{1}+k_{2})-3)!!.
\]
The sum in (\ref{eq:RHS2}) contributes when $k=k_{1}$ and $l=k_{2}-1$:

\[
\frac{K!}{(2k_{1})!(\frac{K}{2}-k_{1})!2^{\frac{K}{2}-k_{1}}}\frac{(L+1)!}{(2k_{2}-2)!(\frac{L}{2}+1-k_{2})!2^{\frac{L}{2}+1-k_{2}}}(2(k_{1}+k_{2})-3)!!.
\]
The sum in (\ref{eq:RHS3}) contributes when $k=k_{1}$ and $l=k_{2}$:

\[
\frac{K!}{(2k_{1})!(\frac{K}{2}-k_{1})!2^{\frac{K}{2}-k_{1}}}\frac{(L+1)!}{(2k_{2})!(\frac{L}{2}-k_{2})!2^{\frac{L}{2}-k_{2}}}(2(k_{1}+k_{2})-1)!!.
\]

The common factor among the three terms on the right hand side is

\[
\frac{K!}{(2k_{1}-1)!(\frac{K}{2}-k_{1})!2^{\frac{K}{2}-k_{1}}}\frac{(L+1)!}{(2k_{2}-2)!(\frac{L}{2}-k_{2})!2^{\frac{L}{2}-k_{2}}}(2(k_{1}+k_{2})-3)!!.
\]

Since this factor is contained in the term on the left hand side,
as well, we can divide both sides by it. The following term remains
on the left:

\[
\frac{1}{2k_{1}}\frac{L+2}{2k_{2}(2k_{2}-1)(L/2+1-k_{2})2}(2(k_{1}+k_{2})-1).
\]

On the right, we get 
\[
\frac{1}{(2k_{2}-1)(\frac{L}{2}+1-k_{2})2},
\]

\[
\frac{1}{2k_{1}(\frac{L}{2}+1-k_{2})2},
\]

\[
\frac{2(k_{1}+k_{2})-1}{2k_{1}2k_{2}(2k_{2}-1)}.
\]

We need to show that both sides are equal. We start by multiplying
both sides by $2k_{1}2k_{2}(2k_{2}-1)(L/2+1-k_{2})2$ and calculate

\begin{align*}
(L+2)(2(k_{1}+k_{2})-1) & \overset{?}{=}2k_{1}2k_{2}+2k_{2}(2k_{2}-1)+(2(k_{1}+k_{2})-1)(L/2+1-k_{2})2\\
(L+2)(2(k_{1}+k_{2})-1) & \overset{?}{=}2k_{1}2k_{2}+2k_{2}(2k_{2}-1)+(2(k_{1}+k_{2})-1)(L+2-2k_{2})\\
0 & \overset{?}{=}2k_{1}2k_{2}+2k_{2}(2k_{2}-1)+(2(k_{1}+k_{2})-1)(-2k_{2})\\
0 & \overset{?}{=}2k_{1}2k_{2}+4k_{2}^{2}-2k_{2}-4k_{1}k_{2}-4k_{2}^{2}+2k_{2}=0.
\end{align*}

This concludes the proof that all terms $\bar{\beta}^{k_{1}+k_{2}}\alpha_{1}^{k_{1}}\alpha_{2}^{k_{2}}$
with $k_{1}\in\{1,\ldots,K/2\}$ and $k_{2}\in\{1,\ldots,L/2\}$ have
equal coefficients on both sides of the formula. We still need to
show the same for the marginal cases where $k_{1}=0$ or $k_{2}\in\{0,L/2+1\}$.
In these five cases, on the right hand side, only one or two of the
sums in (\ref{eq:RHS1}), (\ref{eq:RHS2}) and (\ref{eq:RHS3}) contribute
to the coefficient of $\bar{\beta}^{k_{1}+k_{2}}\alpha_{1}^{k_{1}}\alpha_{2}^{k_{2}}$
. The proofs are very similar, we shall therefore limit ourselves
to the case where $k_{1}=0$ and $k_{2}=0$.

On the left hand side, we obtain the coefficient

\[
\frac{K!}{(K/2)!2^{K/2}}\frac{(L+2)!}{(L/2+1)!2^{L/2+1}}.
\]

On the right hand side, only (\ref{eq:RHS3}) contributes to the coefficient,
as in (\ref{eq:RHS1}) and (\ref{eq:RHS2}) the powers of $\alpha_{1}$
and $\alpha_{2}$ can never both be 0. Hence the coefficient on the
right is

\[
\frac{K!}{(K/2)!2^{K/2}}\frac{(L+1)!}{(L/2)!2^{L/2}}.
\]

Dividing both coefficients by $\frac{K!}{(K/2)!2^{K/2}}\frac{(L+1)!}{(L/2)!2^{L/2}}$,
we obtain 1 on the right and $\frac{L+2}{(L/2+1)2}=1$ on the left.

Since the coefficients on both sides are equal for all possible powers
of $\alpha_{1}$ and $\alpha_{2}$, we conclude that the recursive
formula holds for $M_{K,L}$ This concludes the proof of theorem \ref{CLT_hom}.

\subsubsection{Sublinear Population Growth}

In this section we shall analyse the limiting distribution of the
sums $(S_{1},S_{2})_{N}$ if one or both groups grow at lower rates
than the overall population $N$. As in the previous section, we allow
for the presence of a remainder population, i.e. $N_{1}+N_{2}\leq N$,
where equality need not hold.

Assume again that $\alpha_{1}=\lim\frac{N_{1}}{N}$ and $\alpha_{2}=\lim\frac{N_{2}}{N}$.
Now we allow one or both of these limits to be 0: let $\alpha_{1}=0$
and $0\leq\alpha_{2}\leq1$.

We shall show the following results: $S_{1}$ is asymptotically standard
normal, even though for finite $N$ the variance is of course strictly
greater than 1. In the large $N$ limit, $S_{1}$ and $S_{2}$ become
independent. Note that it suffices for this that one of the two groups
grow more slowly than at linear speed. Hence, contrary to the last
section, where we had positive moments for $K,L$ odd (such as the
covariance, e.g.), here only moments for $K,L$ even are positive.

We already know that only multiindices $\underline{i}=(i_{1},\ldots i_{K})$
with the property that each of the indices occurs only once or twice
contribute to the asymptotic moments. The moments are given by $\mathbb{E}(S_{1}^{K})=$

\begin{align*}
\approx & \frac{1}{N_{1}^{K/2}}\sum_{k=0}^{K/2}\frac{N_{1}!K!}{(2k)!(K/2-k)!(N_{1}-K/2+k)!2^{K/2-k}}(2k-1)!!\bar{\beta}^{k}N^{-k}\\
\approx & \frac{1}{N_{1}^{K/2}}\sum_{k=0}^{K/2}\frac{N_{1}^{K/2+k}K!}{(2k)!(K/2-k)!2^{K/2-k}}(2k-1)!!\bar{\beta}^{k}N^{-k}\\
\approx & \sum_{k=0}^{K/2}\frac{K!}{(2k)!(K/2-k)!2^{K/2-k}}(2k-1)!!\bar{\beta}^{k}(\frac{N_{1}}{N})^{k}\\
\approx & \sum_{k=0}^{K/2}\frac{K!}{(2k)!(K/2-k)!2^{K/2-k}}(2k-1)!!\bar{\beta}{}^{k}\alpha_{1}^{k}\\
\approx & \frac{K!}{(K/2)!2^{K/2}}\\
= & (K-1)!!
\end{align*}

The last approximate equality is due to $\alpha_{1}=0$; only the
summand with $k=0$ contributes asymptotically to the value of the
moment $M_{K}$. This shows that $S_{1}$ tends to a standard normal
distribution.

The bivariate moment $M_{K,L}=\mathbb{E}(S_{1}^{K}S_{2}^{L})$ is
similar:

\begin{align*}
 & \frac{1}{N_{1}^{K/2}}\frac{1}{N_{2}^{L/2}}\sum_{k=0}^{K/2}\sum_{l=0}^{L/2}\frac{N_{1}!K!}{(2k)!(K/2-k)!(N_{1}-K/2+k)!2^{K/2-k}}\cdot\\
 & \cdot\frac{N_{2}!L!}{(2l)!(L/2-l)!(N_{2}-L/2+l)!2^{L/2-l}}(2(k+l)-1)!!\bar{\beta}^{k+l}N^{-(k+l)}\\
\approx & \frac{1}{N_{1}^{K/2}}\frac{1}{N_{2}^{L/2}}\sum_{k=0}^{K/2}\sum_{l=0}^{L/2}\frac{N_{1}^{K/2+k}K!}{(2k)!(K/2-k)!2^{K/2-k}}\cdot\\
 & \cdot\frac{N_{2}^{L/2+l}L!}{(2l)!(L/2-l)!2^{L/2-l}}(2(k+l)-1)!!\bar{\beta}^{k+l}N^{-(k+l)}\\
\approx & \sum_{k=0}^{K/2}\sum_{l=0}^{L/2}\frac{K!}{(2k)!(K/2-k)!2^{K/2-k}}\frac{L!}{(2l)!(L/2-l)!2^{L/2-l}}\cdot\\
 & \cdot(2(k+l)-1)!!\bar{\beta}^{k+l}\alpha_{1}^{k}\alpha_{2}^{l}\\
\approx & \frac{K!}{(K/2)!2^{K/2}}\sum_{l=0}^{L/2}\frac{L!}{(2l)!(L/2-l)!2^{L/2-l}}(2l-1)!!\bar{\beta}^{l}\alpha_{2}^{l}.
\end{align*}

Depending on whether $\alpha_{2}=0$ or $\alpha_{2}>0,$ we continue
in the former case

\[
M_{K,L}=(K-1)!!(L-1)!!,
\]

which shows that for $\alpha_{2}=0$, $(S_{1},S_{2})$ follows an
independent bivariate normal distribution with both variances equal
to 1.

In the latter case, we obtain

\begin{align*}
M_{K,L} & =(K-1)!!\sum_{l=0}^{L/2}\frac{L!}{(2l)!(L/2-l)!2^{L/2-l}}(2l-1)!!\bar{\beta}^{l}\alpha_{2}^{l}\\
 & =(K-1)!!\sum_{l=0}^{L/2}\frac{L!}{(2l)!(L/2-l)!2^{L/2-l}}\frac{(2l)!}{l!2^{l}}(\alpha_{2}\bar{\beta})^{l}\\
 & =(K-1)!!\frac{L!}{(L/2)!2^{L/2}}\sum_{l=0}^{L/2}\frac{(L/2)!}{l!(L/2-l)!}(\alpha_{2}\bar{\beta})^{l}\\
 & =(K-1)!!(L-1)!!(1+\alpha_{2}\bar{\beta})^{L/2},
\end{align*}
which implies a limiting bivariate normal distribution with zero mean
and covariance matrix $\left[\begin{array}{cc}
1 & 0\\
0 & 1+\alpha_{2}\bar{\beta}
\end{array}\right]$. As mentioned previously, it is enough if one of the two groups grows
more slowly for asymptotic independence to occur.

If we inspect the formula for odd $K,L$, $M_{K,L}=$

\begin{align*}
\sum_{k=0}^{\frac{K-1}{2}}\sum_{l=0}^{\frac{L-1}{2}}\frac{K!}{(2k+1)!(\frac{K-1}{2}-k)!2^{\frac{K-1}{2}-k}}\frac{L!}{(2l+1)!(\frac{L-1}{2}-l)!2^{\frac{L-1}{2}-l}}\cdot\\
\cdot(2(k+l)+1)!!\bar{\beta}^{k+l+1}\alpha_{1}^{k+1/2}\alpha_{2}^{l+1/2},
\end{align*}

we notice that even for $k=0$, the power of $\alpha_{1}$ is strictly
positive, and therefore the moment disappears.

\subsection{Remarks on $\beta=1$}

For $\beta=1$, the limiting moments can be calculated as follows:

\begin{equation}
\frac{1}{N_{1}^{3K/4}N_{2}^{3L/4}}\sum_{k=0}^{K}\sum_{\underline{i}\in\Pi_{k}^{(K)}}\sum_{l=0}^{L}\sum_{\underline{j}\in\Pi_{l}^{(L)}}w_{K}(\underline{i})w_{L}(\underline{j})\mathbb{E}(X_{\underline{i}}Y_{\underline{j}}).\label{eq:sum-1}
\end{equation}

\begin{thm}
Let $\beta=1$. Then the expression in (\ref{eq:sum-1}) is asymptotically

\begin{equation}
12^{\frac{K+L}{4}}\frac{\Gamma(\frac{K+L+1}{4})}{\Gamma(\frac{1}{4})}\alpha_{1}^{K/4}\alpha_{2}^{L/4}.\label{eq:MomentsBeta1}
\end{equation}
\end{thm}

\begin{proof}
We calculate for any $\underline{i}\in\prod_{k}^{(K)},\underline{j}\in\prod_{l}^{(L)}$:

\begin{align*}
\left|\mathbb{E}(X(\underline{i},\underline{j}))\right| & \leq c\frac{1}{N^{\frac{k+l}{4}}},\\
w_{K}(\underline{i})w_{L}(\underline{j}) & \leq cN_{1}^{\frac{k+K}{2}}N_{2}^{\frac{l+L}{2}}.
\end{align*}

The symbol $c$ in the above inequalities stands for some constant
(not necessarily the same in both lines).

Therefore, each summand is bounded above by

\begin{align*}
\frac{1}{N_{1}^{3K/4}N_{2}^{3L/4}}w_{K}(\underline{i})w_{L}(\underline{j})\left|\mathbb{E}(X(\underline{i},\underline{j}))\right| & \leq c\frac{1}{N_{1}^{3K/4}N_{2}^{3L/4}}N_{1}^{\frac{k+K}{2}}N_{2}^{\frac{l+L}{2}}\frac{1}{N^{\frac{k+l}{4}}}\\
 & =c\alpha_{1}^{1/4}\alpha_{2}^{1/4}N_{1}^{-\frac{1}{4}(K-k)}N_{2}^{-\frac{1}{4}(L-l)},
\end{align*}

which goes to 0 as $N\rightarrow\infty$ if $K>k$ or $L>l$.

Just as in the one-dimensional case, the only summand that matters
asymptotically is the one where both $k=K,l=L$ hold. In this case,
we have

\begin{align*}
w_{K}(\underline{i}) & \approx N_{1}^{K},\\
w_{L}(\underline{j}) & \approx N_{2}^{L}
\end{align*}

and

\[
\mathbb{E}(X(\underline{i},\underline{j}))\approx12^{\frac{K+L}{4}}\frac{\Gamma(\frac{K+L+1}{4})}{\Gamma(\frac{1}{4})}\frac{1}{N^{\frac{K+L}{4}}}
\]

provided that $K+L$ is even and 0 otherwise. Multiplying these, we
obtain

\[
\frac{1}{N_{1}^{3K/4}N_{2}^{3L/4}}w_{K}(\underline{i})w_{L}(\underline{j})\mathbb{E}(X(\underline{i},\underline{j}))\approx12^{\frac{K+L}{4}}\frac{\Gamma(\frac{K+L+1}{4})}{\Gamma(\frac{1}{4})}\alpha_{1}^{K/4}\alpha_{2}^{L/4}.
\]
\end{proof}
This law of large numbers also implies that the central limit theorem
cannot hold for $\beta\geq1$. However, for $\beta=1$, there is a
bounded measure $\mu$ with the moments given by (\ref{eq:MomentsBeta1}).

email: werner.kirsch@fernuni-hagen.de\qquad{}gabor.toth@fernuni-hagen.de

\begin{thebibliography}{1}
\bibitem{Breiman} Breiman, Leo: Probability, Addison-Wesley 1968

\bibitem{Ellis} Ellis, Richard: Entropy, large deviations, and statistical
mechanics, Whiley 1985

\bibitem{FM} Fedele, Micaela: Rescaled Magnetization for Critical
Bipartite Mean-Fields Models, J. Stat. Phys. 155:223\textendash 226
(2014)

\bibitem{FC} Fedele, Micaela; Contucci, Pierluigi: Scaling Limits
for Multi-species Statistical Mechanics Mean-Field Models, J. Stat.
Phys. 144:1186\textendash 1205 (2011)

\bibitem{key-1}Kirsch, Werner: A Survey on the Method of Moments,
available from http://www.fernuni-hagen.de/stochastik/

\bibitem{key-2}Kirsch, Werner: On Penrose's Square-root Law and Beyond,
Homo Oeconomicus 24(3/4): 357\textendash 380, 2007

\bibitem{key-3}Isserlis, Leon: On a Formula for the Product-Moment
Coefficient of any Order of a Normal Frequency Distribution in any
Number of Variables, Biometrika, Vol. 12, No. 1/2 (Nov., 1918), pp.
134-139 
\end{thebibliography}
\end{document}